\documentclass[oneside,english]{amsart}
\usepackage[T1]{fontenc}
\usepackage[latin9]{inputenc}
\usepackage{color}
\usepackage{array}
\usepackage{float}
\usepackage{amsthm}
\usepackage{amstext}
\usepackage{amssymb}
\usepackage{esint}

\makeatletter

\numberwithin{equation}{section}
\numberwithin{figure}{section}
\theoremstyle{plain}
\newtheorem{thm}{\protect\theoremname}[section]
  \theoremstyle{definition}
  \newtheorem{defn}[thm]{\protect\definitionname}
  \theoremstyle{plain}
  \newtheorem{lem}[thm]{\protect\lemmaname}
  \theoremstyle{remark}
  \newtheorem{rem}[thm]{\protect\remarkname}
  \theoremstyle{plain}
  \newtheorem{assumption}[thm]{\protect\assumptionname}
  \theoremstyle{plain}
  \newtheorem{cor}[thm]{\protect\corollaryname}
  \theoremstyle{plain}
  \newtheorem{prop}[thm]{\protect\propositionname}
  \theoremstyle{definition}

\usepackage{appendix}

\makeatother

\usepackage{babel}
  \providecommand{\assumptionname}{Assumption}
  \providecommand{\corollaryname}{Corollary}
  \providecommand{\definitionname}{Definition}
  \providecommand{\lemmaname}{Lemma}
  \providecommand{\problemname}{Problem}
  \providecommand{\propositionname}{Proposition}
  \providecommand{\remarkname}{Remark}
\providecommand{\theoremname}{Theorem}

\DeclareMathOperator{\cpct}{Cap}

\begin{document}

\title{Credit default prediction and parabolic potential theory}

\author{Matteo L. Bedini$^1$}
\address{$^1$Intesa Sanpaolo, Milano, Italy}
\email{matteo.bedini@intesasanpaolo.com}

\author{Michael Hinz$^2$}
\address{$^2$Fakult\"at f\"ur Mathematik, Universit\"at Bielefeld, Postfach 100131, 33501 Bielefeld, Germany}
\email{mhinz@math.uni-bielefeld.de}

\thanks{Part of this work has been financially supported by the European Community's FP 7 Program under contract PITN-GA-2008-213841, Marie Curie ITN \flqq Controlled Systems\frqq.}

\date{\today}
\begin{abstract}
We consider an approach to credit risk in which the information about the time of bankruptcy is modelled using a Brownian bridge that starts at zero 
and is conditioned to equal zero when the default occurs.
This raises the question whether the default can be foreseen by observing the
evolution of the bridge process. 
Unlike in most standard models for credit risk, we allow the
distribution of the default time to be singular. Using a well known fact 
from parabolic potential theory, we provide a sufficient condition for its
predictability. 
\end{abstract}

\keywords{Default time, Predictable Stopping Time, Brownian Bridge on Random
Intervals, Riesz Capacity, Hausdorff dimension.}

\maketitle

\section{Introduction}

\noindent 

\noindent We consider an information-based approach to credit risk where the flow of information concerning a default time $\tau$ is modelled
explicitly through the natural completed filtration $\mathbb{F}^\beta$ generated by
the underlying \textit{information process} $\beta=\left(\beta_{t},\, t\geq0\right)$.
In our model this process is defined to be a Brownian bridge starting and ending at zero 
on the random time interval $\left[0,\tau\right]$,
\begin{equation}\label{E:firstformula}
\beta_{t}=W_{t}-\frac{t}{\tau\vee t}W_{\tau\vee t},\: t\geq0,
\end{equation}
where $W=\left(W_{t},\, t\geq0\right)$ is a Brownian motion independent
of $\tau$. The presence of the bridge process in a neighborhood
of zero models periods of danger of an imminent default. In a mathematical
model for credit risk it is important to know whether the default
time is a predictable or a totally inaccessible stopping time. In
the first case the default should be foreseen by market agents (something
that, in general, does not happen). In the second case the default
occurs by surprise, a fact that is widely accepted. Models used by practitioners assume that the law $\mathbf{P}_{\tau}$  
of the default time $\tau$ admits a density with respect to the Lebesgue
measure, and in many cases this condition is sufficient to guarantee the total inaccessibility of $\tau$. We refer to \cite{key-6} or \cite{key-7} for a discussion of the role of information in credit risk models and its relation to predictability properties of the default time. 
 
Here we allow
the law $\mathbf{P}_{\tau}$ to be singular (i.e. such that no density exists). Then for our model the question of predictability is much less clear and not covered by the existing literature. We provide a sufficient condition for predictability in terms of the size of the support of $\mathbf{P}_{\tau}$. In view of the above remarks this could be interpreted as a statement about the scope and limitations of models involving singular laws. As mentioned above, we work under the following standing assumption.

\begin{assumption}\label{A:basic}
\label{ass:Assumption} The random time $\tau$ and the Brownian motion
$W$ in (\ref{E:firstformula}) are independent.
\end{assumption}

Let $\Gamma$ denote the support of the law $\mathbf{P}_{\tau}$. Our main observation, Theorem \ref{TEO:predictable tau1}, states that if
$\Gamma$ is 'small' and separated away from zero, then the default time $\tau$ is an $\mathbb{F}^\beta$-predictable stopping time, i.e. the credit event can be foreseen. 
By $\cpct_{\frac{1}{2}}(\Gamma)$ we denote the Riesz capacity
of order $\frac{1}{2}$ of $\Gamma$.

\begin{thm}
\label{TEO:predictable tau1}
Let Assumption \ref{A:basic} be satisfied, assume that $0\notin\Gamma$ and that for all $T>0$ we have $\cpct_{\frac{1}{2}}\left(\Gamma\cap\left[0,T\right]\right)=0$. Then $\tau$ is a predictable $\mathbb{F}^{\beta}$-stopping time.
\end{thm}

From Frostman's Lemma, see e.g. 
\cite[Chapter 8, Theorem 8.8]{key-14} or \cite[Section 4.3]{key-15}, we know that $\cpct_{\frac{1}{2}}\left(\Gamma\cap\left[0,T\right]\right)=0$  whenever the $\frac12$-dimensional Hausdorff measure $\mathcal{H}^{\frac12}(\Gamma)$ of $\Gamma$ is finite. A sufficient condition is that the Hausdorff dimension $\textrm{dim}_H \Gamma$ of $\Gamma$ is less than $\frac12$.

\begin{cor} Let Assumption \ref{A:basic} be satisfied.
If $\mathcal{H}^{\frac12}(\Gamma)<+\infty$, then $\tau$ is a predictable stopping time. In particular, any default time $\tau$ with countable support $\Gamma$ is $\mathbb{F}^\beta$-predictable. 
\end{cor}

Theorem \ref{TEO:predictable tau1} is a straightforward application to credit risk of a potential theoretic result of Khoshnevisan and Xiao, \cite[Proposition 1.4]{key-13}, linking thermal capacities and hitting probabilities for Brownian motion. In order to apply it we use key results from  \cite{key-2}. In particular, we employ a novel definition
of the information process, \cite[Definition 3.1]{key-2} and a key result on conditional
expectation of functionals of $\beta$ with respect
to the $\sigma$-algebra generated by the default time $\tau$.

The idea of modelling the information about the default time with
a Brownian bridge on a stochastic interval was introduced
in \cite{key-1}. The definition of the information process
$\beta$, the study of its basic properties and an application to
the problem of pricing a credit derivative instrument appeared in
\cite{key-2}. In the forthcoming paper, \cite{key-3},
 the reader will find sufficient conditions for making $\tau$ a totally
inaccessible stopping time with respect to $\mathbb{F}^{\beta}$ and,
moreover, the explicit formula for the compensator process
$K$ of the $\mathbb{F}^{\beta}$-submartingale $\left(\mathbb{I}_{\left\{ \tau\leq t\right\} },\, t\geq0\right)$ (see also \cite[Section 3.2]{key-1}). 

In view of the application to credit risk, a comparison of the information-based framework with the study of the minimal filtration  $\mathbb{H}=\left(\mathcal{H}_{t}\right)_{t\geq0}$
making the positive random variable $\tau$ a stopping time is particularly meaningful. Earlier work of Dellacherie and Meyer, \cite{key-4}, has shown that if the law $\mathbf{P}_{\tau}$ of  $\tau$ is diffuse then $\tau$ is a totally
inaccessible stopping time with respect to $\mathbb{H}$ (see, for example, \cite[Chapter 4, Section 3, Theorem 107]{key-4}). In the special case where $\mathbf{P}_{\tau}$ admits a density with respect to the Lebesgue measure, $\tau$ is an $\mathbb{F}^{\beta}$-totally inaccessible stopping time (see, e.g., \cite[Theorem 3.43 and
Corollary 3.44]{key-1}). If, on the other hand, $\mathbf{P}_{\tau}$ is singular
with respect to the Lebesgue measure, then a classification of $\tau$ with respect
to the filtration $\mathbb{F}^{\beta}$ seems more complicated than the classification
with respect to the filtration $\mathbb{H}$. Theorem \ref{TEO:predictable tau1} may be considered as a first result in this direction.

The paper is organized as follows. Section 2 contains a precise definition
of the information process and surveys some results from \cite{key-1} and \cite{key-2}.
In Section 3 we briefly cite the needed facts from \cite{key-13}, and in Section 4 we prove Theorem \ref{TEO:predictable tau1}.

\section{The Information Process}

\noindent Let use denote by $\beta^r = \left(\beta_t ^r ,\,0\leq t\leq r\right)$ a real-valued Brownian bridge  between $0$ and $0$ on a deterministic time interval $\left[0,r\right]$. It is a continuous-time
stochastic process obtained by ``forcing'' a (real valued) Brownian
motion started from $0$ to assume the value $0$
at some given time $r\in\left(0,+\infty\right)$ (see for example \cite[Section 5.6.B]{key-12}). In this section we recall the definition and the main properties of the information process $\beta$, which generalizes the definition of Brownian bridge to the case where the time interval is no longer deterministic. We refer to \cite{key-2} for proofs and further results on Brownian bridges on random intervals. 

Let $\left(\Omega,\mathcal{F},\mathbf{P}\right)$
be a complete probability space and denote by $\mathcal{N}_{P}$ the
collection of $\mathbf{P}$-null sets of $\mathcal{F}$. Let $\tau:\Omega\rightarrow\left(0,+\infty\right)$
be a strictly positive random time. 
The time $\tau$ models the random time at which a credit event takes
place, and we refer to it as \textit{default time}. Let $W=\left(W_{t},\, t\geq0\right)$ be a real-valued Brownian motion over $\left(\Omega,\mathcal{F},\mathbf{P}\right)$, starting from 0. For the remaining part of the paper let Assumption \ref{A:basic} be in effect. We define the Brownian bridge on a stochastic
interval as proposed in \cite[Definition 3.1]{key-2}. 
\begin{defn}
The process $\beta=\left(\beta_{t},\, t\geq0\right)$ given by
\begin{equation}
\beta_{t}:=W_{t}-\frac{t}{\tau\vee t}W_{\tau\vee t},\: t\geq0\label{eq:info-proc}
\end{equation}
will be called \textit{information process}. 
\end{defn}
Defining $\mathcal{F}_{t}^{\beta}:=\bigcap_{n\in\mathbb{N}}\sigma\left(\beta_{s},\,0\leq s\leq t+\frac{1}{n}\right)\vee\mathcal{N}_{P}$ and $\mathbb{F}^{\beta}:=\left(\mathcal{F}_{t}^{\beta}\right)_{t\geq0}$ we obtain the natural (right-continuous and complete) filtration $\mathbb{F}^{\beta}$ of $\beta$. We quote \cite[Corollary 2.2]{key-2}.

\begin{lem}\label{cor:LEM if--is} Let $C$ be the canonical space of real-valued function defined on $\mathbb{R}_+$. If $h:\,\left(0,+\infty\right)\times C\rightarrow\mathbb{R}$
is a measurable function such that $\mathbf{E}\left[|h\left(\tau,\,\beta\right)|\right]<+\infty$,
then 
\[
\mathbf{E}\left[h\left(\tau,\,\beta\right)|\tau=r\right]=\mathbf{E}\left[h\left(r,\,\beta^{r}\right)\right],\;\mathbf{P}_{\tau}\textrm{-a.s.}
\]
and
\[
\mathbf{E}\left[h\left(\tau,\,\beta\right)|\sigma\left(\tau\right)\right]\left(\omega\right)=\left(\mathbf{E}\left[h\left(r,\,\beta^{r}\right)\right]\right)_{r=\tau\left(\omega\right)},\;\mathbf{P}\textrm{-a.s.}
\]
\end{lem}

Lemma \ref{cor:LEM if--is} provides a useful connection between the law of the
Brownian bridge and the conditional law with respect to $\sigma\left(\tau\right)$
of a generic functional involving $\tau$, the Brownian motion $W$
and the process $\beta$ defined in equation (\ref{eq:info-proc}).  In particular, if $0<t<r$, the law of $\beta_{t}$, conditional on $\tau=r$, is
the same of that of $\beta_t ^r$.

For a proof of the following fact see \cite[Proposition 3.1 and Corollary 3.1]{key-2}.
\begin{lem}
\label{lem:For-all-,}For all $t>0$, $\left\{ \beta_{t}=0\right\} =\left\{ \tau\leq t\right\} ,\:\mathbf{P}$-a.s.
In particular, $\tau$ is an $\mathbb{F}^{\beta}$-stopping time.
\end{lem}

\section{\label{sec:Inverse-Images-and}Hitting Probabilities and Capacities for Brownian Bridges}

\noindent Given a compact set $E\subseteq\left(0,+\infty\right)$ and
and an $\mathbb{R}^{n}$-valued standard
Brownian motion $B=\left(B_{t},\, t\geq0\right)$, the \textit{range}
of the set $E$ under $B$ is denoted by
\[
B(E):=\left\{ \left(x,\omega\right)\in\mathbb{R}^{n}\times\Omega:\,\exists t\in E\textrm{ such that }B_{t}\left(\omega\right)=x\right\} \subseteq\mathbb{R}^{n}\times\Omega.
\]
The relevant question is under which conditions on $E$ 
we have 
\begin{equation}
\mathbf{P}\left(B\left(E\right)\cap \left\lbrace 0\right\rbrace\neq\varnothing\right)=0,\label{eq:problem-1}
\end{equation}
and an answer to this question was provided by Khoshnevisan and Xiao in \cite{key-13}. Let $\mathcal{P}_{d}\left(E\times G\right)$ denotes the collection
of all Borel probability measures $\mu$ on $E\times G$ that are \textquotedblleft diffuse\textquotedblright{}
in the sense that $\mu\left(\left\{ t\right\} \times G\right)=0$
for all $t>0.$ For $\mu\in\mathcal{P}_{d}\left(E\times G\right)$
we define the\textit{ parabolic $0$-energy $\mathcal{E}\left(\mu\right)$ of $\mu$}
by 
\begin{equation}\label{E:parabolicenergy}
\mathcal{E}\left(\mu\right):={\displaystyle \intop_{E\times G}\intop_{E\times G}\frac{e^{-\left\Vert x-y\right\Vert ^{2}/\left(2|t-s|\right)}}{|t-s|^{\frac{n}{2}}}\mu\left(ds\times dx\right)\mu\left(dt\times dy\right)}.
\end{equation}
We quote \cite[Proposition 1.4]{key-13}, which implies
a characterization of (\ref{eq:problem-1}) in terms of the parabolic $0$-energy.

\begin{prop}
\label{prop:Suppose--KX}
Suppose $G\subset\mathbb{R}^{n}$ 
is compact and has Lebesgue measure 0. Then $\mathbf{P}\left(B\left(E\right)\cap G\neq\varnothing\right)>0$
if and only if there exists a probability measure $\mu\in \mathcal{P}_d(E\times G)$
such that $\mathcal{E}\left(\mu\right)<+\infty$.\end{prop}

Earlier and related well known results on hitting probabilities can for instance be found in Doob, \cite{key-5}, Kaufmann, \cite{key-10}, Kaufman and Wu, \cite{key-11}, Watson, \cite{key-18}, Xiao, \cite{key-20}, and M\"orters and Peres, \cite[Chapter
8, Section 3]{key-15}.

We need an extension of Proposition \ref{prop:Suppose--KX} to the case of the Brownian bridge. 
For the convenience of the reader we recall the concepts of Riesz energies and capacities, details can be found in \cite[Section 4.3]{key-15}, \cite[Chapter 8]{key-14} or \cite[Section 3.1]{key-20}. For a Radon measure $\nu$ on $(0,+\infty)$ and $s\geq 0$, the
\emph{$s$-energy $\mathcal{I}_{s}\left(\nu\right)$ associated with the measure
$\nu$} is defined by
\[\mathcal{I}_{s}\left(\nu\right):=\intop_0^\infty\intop_0^\infty\left| x-y\right| ^{-s}\nu\left(dx\right)\nu\left(dy\right).\]
Given a set $A\subseteq (0,+\infty)$ let $\mathcal{P}\left(A\right)$ be the set of Borel probability  measures $\nu$ on $A$. For $s>0$, the \emph{Riesz $s$-capacity $\cpct_{s}\left(A\right)$
of $A$} is defined by
\[
\cpct_{s}\left(A\right):=\sup\left\{ \mathcal{I}_{s}\left(\nu\right)^{-1}:\,\nu\in\mathcal{P}\left(A\right)\right\} ,
\]
with the convention that $\cpct_{s}\left(\varnothing\right)=0$. The following corollary is a simple application of the previous proposition to the case $n=1$.

\begin{cor}
\label{cor:capacity-polarity}Let $B=\left(B_{t},\, t\geq0\right)$
be an $\mathbb{R}$-valued Brownian motion.
Then we have
$\mathbf{P}\left(B\left(E\right)\cap \left\lbrace 0\right\rbrace \neq\varnothing\right)=0$
if and only if $\cpct_{\frac{1}{2}}\left(E\right)=0$.
\end{cor}
\begin{proof}
Any
$\mu \in \mathcal{P}_d(E\times \left\lbrace 0\right\rbrace)$ is equal to the product of a finite measure
$\nu$ on $E$ which assigns zero to any single element set $\left\lbrace t\right\rbrace$, $t>0$, and a finite point-mass measure $\delta_{0}$ on $\left\{ 0\right\} $, i.e. $\mu=\nu\otimes \delta_0$. We may assume that both $\nu$
and $\delta_{0}$ are members of $\mathcal{P}\left(E\right)$. The parabolic
0-energy (\ref{E:parabolicenergy}) associated with $\mu$ then equals
\[
\mathcal{E}\left(\mu\right)=\intop_{E}\intop_{E}|t-s|^{-\frac{1}{2}}\nu\left(ds\right)\nu\left(dt\right)=\mathcal{I}_{\frac{1}{2}}\left(\nu\right).
\] 
On the other hand, given a probability measure $\nu\in\mathcal{P}(E)$ 
with $\mathcal{I}_{\frac{1}{2}}\left(\nu\right)<+\infty$ we necessarily have $\nu(\left\lbrace t\right\rbrace)=0$ for any $t>0$, and therefore $\mu:=\nu\otimes\delta_0$ is a member of $\mathcal{P}_d(E\times \left\lbrace 0\right\rbrace)$. Consequently there exists $\mu\in \mathcal{P}_{d}\left(E\times \left\lbrace 0\right\rbrace\right)$ with $\mathcal{E}(\mu)<+\infty$ if and only if $\cpct_{\frac{1}{2}}\left(E\right)>0$, and by Proposition \ref{prop:Suppose--KX} the result follows.
\end{proof}

The above easily carry over to the Brownian bridge case. 
\begin{lem}
\label{lem:result_for_std_BB} Given a closed subset $E$ of $\left(0,r\right]$ set $\gamma^{r}_E\left(\omega\right):=\inf\left\{ t\in E:\,\beta_{t}^{r}\left(\omega\right)=0\right\}$. We have 
$\mathbf{P}\left(\gamma^{r}_E<r\right)=0$ if and only if 
$\cpct_{\frac{1}{2}}\left(E\right)=0$.
\end{lem}

\begin{proof}
Let be $0<T<r$ and let $\mathbf{Q}^{T}$ be a probability measure
equivalent to $\mathbf{P}$ such that $\left(\beta^{r} _t,\,0\leq t\leq T\right)$ is a $\mathbf{Q}^{T}$
Brownian motion on $\left[0,T\right]$ (this is always possible for every $T<r$; it suffices to apply Girsanov theorem after noticing that the process $\left(\frac{\beta^r _t}{r-t},\,0\leq t\leq T\right)$ satisfies the Novikov condition). Let $0<\delta < \min E$. Since 
obviously $\gamma^{r}_E>\delta$, we have $\left\{ \beta^{r}\left(E\cap\left[\delta,T\right]\right)\cap\left\{ 0\right\} \neq\varnothing\right\} =\left\{ \gamma^{r}_E\leq T\right\}$.
By Corollary \ref{cor:capacity-polarity} we have $\cpct_{\frac12}(E\cap\left[\delta,T\right])=0$ if and only if 
$\mathbf{Q}^{T}\left(\gamma^{r}_E\leq T\right)=0$. But as $\mathbf{P}$
is equivalent to $\mathbf{Q}^{T}$, this is the case if and only if $\mathbf{P}\left(\gamma^{r}_E\leq T\right)=0$. Since $\cpct_{\frac12}(E)=\sup_{T\leq r} \cpct_{\frac12}(E\cap\left[\delta,T\right])$ (recall that the Riesz capacities are Choquet capacities), the result now follows by letting $T\uparrow r$.
\end{proof}

\section{\label{sec:PredictableDefTime}Predictable Default Times}

\noindent To fix notation we recall the definition of a predictable stopping
time, see for instance \cite[Chapter 4, Section 3]{key-4}. The \textit{predictable $\sigma$-algebra}
associated to a filtration $\mathbb{F}=\left(\mathcal{F}_{t}\right)_{t\geq0}$,
denoted by $\mathcal{P}\left(\mathbb{F}\right)$, is the $\sigma$-algebra
defined on $\mathbb{R}_{+}\times\Omega$ generated by all $\mathbb{F}$-adapted processes
with left-continuous paths on $\left(0,+\infty\right)$. 
A positive random variable $T$ is a called a \textit{predictable stopping time with respect to $\mathbb{F}$}
if the (random) subset of $\mathbb{R}_{+}\times\Omega$ given by
$\left\{ \left(t,\omega\right):\, T\left(\omega\right)\leq t<+\infty\right\}$ is measurable with respect to $\mathcal{P}\left(\mathbb{F}\right)$. 
If $T$ is a predictable stopping time, then there exists a sequence $T_n,\,n\geq0$ of stopping times (the so-called \textit{announcing sequence}) converging to $T$ from below (see \cite[Theorem N. 76]{key-4}). 
For example, the first entry time $T$ of a continuous process $X$ in a closed set $C$ is predictable and announced by the sequence of stopping times where the distance of $X$ from $C$ is equal to $\frac{1}{n}$.

By $\Gamma:=\textrm{supp}\:\mathbf{P}_{\tau}$ we denote the support of the law $\mathbf{P}_{\tau}$ of the default time $\tau$. It is clear that $\Gamma$ is closed and
$\Gamma\cap\left[\delta,T\right]$ is compact for every $T,\delta>0$
such that $T>\delta$. We now consider the two-dimensional process $X=\left(X_{t},\, t\geq0\right)$
given by 
\[
X_{t}:=\left[\begin{array}{c}
\textrm{dist}_\Gamma\left(t\right)\\
\beta_{t},
\end{array}\right],\]
where $\textrm{dist}_\Gamma\left(t\right):=\min_{r\in\Gamma}|r-t|$ denotes the Euclidean distance of $t\in\mathbb{R}_+$ to $\Gamma$. Clearly the process $X$ is continuous and $\mathbb{F}^{\beta}$-adapted. By
$\gamma^X_0$ we denote the first hitting time of zero of the process
$X$, $\gamma^X_0:=\inf\left\{ t>0:\, X_{t}=0\right\}$.

\begin{rem}
\label{rem:Since--is}\mbox{}
\begin{enumerate}
\item[(i)] Since $X$ is continuous and adapted to the 
filtration $\mathbb{F}^{\beta}$, its first hitting time $\gamma^X_0$
of the closed set $\left\{ 0\right\} $ is predictable with respect to $\mathbb{F}^{\beta}$.
\item[(ii)] Since $X_{\tau}=0$ by definition of $X$, it follows that $\gamma_0^X\leq\tau$.
\end{enumerate}
\end{rem}

In view of this remark, Theorem \ref{TEO:predictable tau1} results from the following. 

\begin{thm}
\label{TEO:predictable tau}
Assume that $0\notin\Gamma$ and that for all $T>0$ we have $\cpct_{\frac{1}{2}}\left(\Gamma\cap\left[0,T\right]\right)=0$. Then we have $\tau=\gamma_0^X$
$\mathbf{P}$-a.s. 

\end{thm}
\begin{proof}
By Remark \ref{rem:Since--is} (ii) it suffices to show that $\mathbf{P}\left(\gamma_0^X<\tau\right)=0$. Recall the notation $\gamma_\Gamma^r=\inf\left\{ t\in\Gamma:\,\beta_{t}^{r}=0\right\}$. Conditionally on $\tau=r$, the process $\beta$ has the same law as the (extended)
Brownian bridge $\beta^{r}$. Consequently 
\[\mathbf{P}\left(\gamma_0^X<\tau\right) =\intop_{\left(0,+\infty\right)}\mathbf{P}\left(\gamma_0^X<\tau|\tau=r\right)\mathbf{P}_{\tau}\left(dr\right) =\intop_{\Gamma}\mathbf{P}\left(\gamma^{r}_\Gamma<r\right)\mathbf{P}_{\tau}\left(dr\right),\]
where the second equality is a consequence of Lemma \ref{cor:LEM if--is}. Using Lemma \ref{lem:result_for_std_BB},
we immediately obtain $\mathbf{P}\left(\gamma^{r}_\Gamma<r\right)=0$,
implying that $\mathbf{P}\left(\gamma_0^X<\tau\right)=0$.
\end{proof}

\end{document}